\documentclass{article}

\usepackage{amsmath,amssymb}
\usepackage{amsthm}
\usepackage{enumitem}
\usepackage{cite}

\usepackage{tikz}
\tikzset{
every node/.style={draw, circle, inner sep=2pt}
}

\usepackage{soul}
\usepackage{cancel}
\newcommand{\rbf}[1]{\textbf{\color{red}#1}}

\newtheorem{theorem}{Theorem}[section]

\newtheorem{proposition}[theorem]{Proposition}
\newtheorem{corollary}[theorem]{Corollary}

\theoremstyle{definition}
\newtheorem{definition}[theorem]{Definition}

\newtheorem{remark}[theorem]{Remark}
\newtheorem{example}[theorem]{Example}

\newcommand{\trans}{^\top}

\newcommand{\oml}{{\bf m}}
\newcommand{\mat}{\operatorname{Mat}_n(\mathbb{R})}
\newcommand{\msym}{\operatorname{Sym}_n(\mathbb{R})}
\newcommand{\mskew}{\operatorname{Skew}_n(\mathbb{R})}
\newcommand{\calB}{\mathcal{B}}
\newcommand{\Rmn}{\mat}
\newcommand{\Rsn}{\msym}
\newcommand{\Rskewn}{\mskew}

\newcommand{\qual}{\mathcal{Q}}
\newcommand{\sptn}{P}
\newcommand{\mptn}{\mathcal{S}}
\newcommand{\mptncl}{\mathcal{S}^{\rm cl}}
\newcommand{\qualv}{\mathcal{Q}^{\rm v}}

\newcommand{\bzero}{{\bf 0}}
\newcommand{\bc}{{\bf c}}
\newcommand{\bd}{{\bf d}}

\newcommand{\bu}{{\bf u}}

\newcommand{\bw}{{\bf w}}

\newcommand{\rank}{\operatorname{rank}}

\newcommand{\range}{\operatorname{range}}
\newcommand{\spec}{\operatorname{spec}}
\newcommand{\tr}{\operatorname{tr}}
\newcommand{\vspan}{\operatorname{span}}
\newcommand{\piner}{\operatorname{pin}}
\newcommand{\iner}{\operatorname{in}}
\newcommand{\riner}{\operatorname{rin}}
\newcommand{\inp}[2]{\left\langle#1,#2\right\rangle}

\title{The bifurcation lemma for strong properties in the inverse eigenvalue problem of a graph}

\author{
Shaun M. Fallat
\thanks{Department of Mathematics and Statistics, University of Regina, Regina, SK, S4S 0A2,
Canada (shaun.fallat@uregina.ca).}
\and
H. Tracy Hall
\thanks{Hall Labs, LLC (h.tracy@gmail.com)}
\and
Jephian C.-H.~Lin
\thanks{Department of Applied Mathematics, National Sun Yat-sen University, Kaohsiung 80424, Taiwan (jephianlin@gmail.com)}
\and 
Bryan L. Shader
\thanks{Department of Mathematics, University of Wyoming, Laramie, WY 82071, USA (bshader@uwyo.edu)}
}
\date{\today}

\begin{document}

\maketitle

\begin{abstract}
The inverse eigenvalue problem of a graph studies the real symmetric matrices whose off-diagonal pattern is prescribed by the adjacencies of the graph.
The strong spectral property (SSP) is an important tool for this problem.
This note establishes the bifurcation lemma, which states that if a spectrum can be realized by a matrix with the SSP for some graph, then all the nearby spectra can also be realized by matrices with the SSP for the same graph.  The idea of the bifurcation lemma also works for other strong properties and for not necessarily symmetric matrices.  This is used to develop new techniques for verifying a spectrally arbitrary pattern or inertially arbitrary pattern.  The bifurcation lemma provides a unified theoretical foundation for several known results, such as the stable northeast lemma and the nilpotent-centralizer method.
\end{abstract}  


\medskip


\section{Introduction}
Let $G$ be a graph on $n$ vertices.  Define $\mptn(G)$ as the set of all $n\times n$ real symmetric matrices whose $(i,j)$-entry, for $i\neq j$, is nonzero if and only if $\{i,j\}$ is an edge.  Note that the diagonal entries can be any real number.  The inverse eigenvalue problem of a graph $G$ (IEP-$G$) asks what spectra can be realized over all matrices in $\mptn(G)$? 

The IEP-$G$ has become a central study for researchers interested in general aspects of the spectra of graphs or certain branches of combinatorial matrix theory. From the early beginnings of Parter and Weiner (see \cite{Parter60,Wiener84}), studying the behavior of eigenvalue multiplicities in trees and particular induced subgraphs, with mild extensions to include cycles, this subject has continued to expand, although a complete resolution seems rather distant. However, over the past few years, motivated in part by seminal work of Colin de Veredi\`ere \cite{CdV} incorporating a particular aspect of genericity (called the strong Arnol'd property), other similar properties have been developed with significant impact on recent advances to the IEP-$G$. These matrix properties, known as the strong spectral and strong multiplicity properties (see \cite{IEPG2, gSAP}), are defined below and have been used with considerable success on specific aspects of the IEP-$G$. These ``strong" properties allow one to preserve some spectral information under certain matrix perturbations, which ability has allowed new sorts of analysis on the IEP-$G$ by considering forbidden minor characterizations and some extremal problems involving related parameters including: the minimum rank, maximum multiplicity, and the minimum number of distinct eigenvalues of a graph.  See, e.g., \cite{IEPG2, gSAP} and the references therein.

A natural question is the following.  Let $G$ be a graph on $n$ vertices.  Suppose $A\in\mptn(G)$ has $\spec(A) = \Lambda$.  For a sequence $\Lambda'$ of $n$ real numbers such that $\|\Lambda - \Lambda'\|$ (treated as vectors) is small enough, is there a matrix $A'\in \mptn(G)$ such that $\spec(A') = \Lambda'$?   

The purpose of this note is to provide a positive answer to the above question by proving the so-called bifurcation lemma, utilizing the inverse function theorem.  The organization of the paper is as follows:  In Section~\ref{sec:invft}, we develop a version of the inverse function theorem for the purpose of proving the bifurcation lemma.  The statement of the bifurcation lemma and its proof will be given in Section~\ref{sec:bifur}, which leads to an immediate consequence known in \cite{AHLvdH13} as the stable northeast lemma.  Applications of the bifurcation lemma on the number of minimum distinct eigenvalues and on the ordered multiplicity lists of a cycle will be given in Section~\ref{sec:app}.  Finally, a version of the bifurcation lemma for non-symmetric matrices is given in Section~\ref{sec:nsym} and, as a consequence, a more flexible version of the nilpotent-centralizer method \cite{GB13} is presented and new methods for finding spectrally arbitrary and inertially arbitrary patterns are developed.  

In the remainder of this section, we set up some basic notation of matrices and review the definition and some propositions of the strong properties.  Let $A$ be a symmetric matrix and $\spec(A) = \{\lambda_1^{(m_1)}, \ldots, \lambda_q^{(m_q)}\}$ with $\lambda_1 < \cdots < \lambda_q$.  The \emph{ordered multiplicity list} of $A$ is the sequence $\oml(A) = (m_1,\ldots, m_q)$.  The number of distinct eigenvalues of $A$ is denoted as $q(A) = q$.  Let $n_+(A)$, $n_-(A)$, and $n_0(A)$ be the number of positive, negative, and zero eigenvalues of $A$, respectively.  Then the \emph{partial inertia} of $A$ is the pair $\piner(A) = (n_+(A), n_-(A))$.  When $A$ is an $n\times n$ matrix, the partial inertia is enough to determine the inertia $(n_+(A), n_-(A), n_0(A))$ since $n_+(A) + n_-(A) + n_0(A) = n$.  
Assuming $n$ if fixed, it follows that the partial inertia is equivalent to the inertia, while the partial inertia is more convenient for our purposes and is used in the literature, e.g., \cite{AHLvdH13,BHL09}.
The matrix norm used in this paper is the Frobenius norm $\|A\|^2 = \tr(A\trans\!A)$.

\subsection{Strong properties}
Let $\Rmn$ be the family of $n\times n$ real matrices.    Let $\calB_n$ be the set of matrices in $\Rmn$ with $\|A\|< 0.5$, noting that for all $L\in\calB_n$
the matrix $I + L$ is invertible.
Let $\Rsn$ and $\Rskewn$ be the families of real symmetric and real skew-symmetric matrices of order $n$, respectively.  Throughout the paper, we will view them as vector spaces equipped with the inner product
\[\inp{A}{B} = \tr(B\trans\!A).\]

Finally, for a given graph $G$, let $\mptncl(G)$ be the topological closure of $\mptn(G)$.  That is, 
\[\mptncl(G) = \{A = \begin{bmatrix} a_{ij} \end{bmatrix}\in \Rsn : a_{ij} = 0 \text{ if }ij\in E(\overline{G})\},\]
which is a subspace of $\Rsn$.  

Let $A\in\Rsn$ be a matrix and $q$ the number of distinct eigenvalues of $A$.  Then $A$ has the \emph{strong spectral property} (or the SSP) if $X = O$ is the only symmetric matrix satisfying 
\[A\circ X = O,\quad I \circ X = O,\quad \text{and }[A,X] = O.\]
Here $[A,X] = AX - XA$.  The matrix $A$ has the \emph{strong multiplicity property} (or the SMP) if $X = O$ is the only symmetric matrix satisfying 
\[A\circ X = O,\quad I \circ X = O,\quad [A,X] = O,\]
\[\text{and }\tr(A^kX) = 0\text{ for }k = 0,\ldots, q-1.\]
Finally, $A$ has the \emph{strong Arnol'd property} (or the SAP) if $X = O$ is the only symmetric matrix satisfying 
\[A\circ X = O,\quad I \circ X = O,\quad \text{and }AX = O.\]
The SAP was used in \cite{CdVF} (English translation \cite{CdV}) to define the Colin de Verdi\`ere parameter $\mu(G)$.  The SSP and the SMP were introduced in \cite{gSAP} for the study of the IEP-$G$.

Let $U$ and $W$ be vector subspaces of $\Rsn$.  Then $U\cap W = \{\bzero\}$ if and only if $U^\perp + W^\perp = \Rsn$.  We will use this fact to produce an equivalent condition of the strong properties mentioned above. 

Suppose $A$ is a matrix in $\mptn(G)$ with $q$ distinct eigenvalues.  We may calculate the orthogonal complements for the following subspaces.  Let
\[U = \{X\in\Rsn: A\circ X = O,\ I\circ X = O\}.\]
Then $U^\perp = \mptncl(G)$.  Additionally, for 
\[\begin{aligned}
 W_S &= \{X\in\Rsn: [A,X] = O\}, \\
 W_M &= \{X\in\Rsn: [A,X] = O\text{ and }\tr(A^kX) = 0\text{ for }k = 0,\ldots, q-1\}, \\
 W_A &= \{X\in\Rsn: AX = O\}, 
\end{aligned}\]
we have 
\[\begin{aligned}
 W_S^\perp &= \{K\trans\!A + AK: K \in \Rskewn\}, \\
 W_M^\perp &= \{K\trans\!A + AK: K\in\Rskewn\} + \vspan\{A^k: k = 0,\ldots, q-1\}, \\
 W_A^\perp &= \{L\trans\!A + AL: L \in \Rmn\}.\\
\end{aligned}\]
The calculation of these orthogonal complements can be found in, e.g., \cite[Theorem~27]{gSAP}.  (In fact, these are the tangent and normal spaces of certain manifolds, but in this note we demonstrate the result using elementary language.)
Thus, we have the following proposition.

\begin{proposition}
\label{prop:sxptan}
Let $G$ be a graph and $A\in\mptn(G)$ a matrix with $q$ distinct eigenvalues.  Then 
\begin{enumerate}
\item[{\rm (a)}] The matrix $A$ has the SSP if and only if 
\[\mptncl(G) +  \{K\trans\!A + AK: K\in\Rskewn\} = \Rsn.\]
\item[{\rm (b)}] The matrix $A$ has the SMP if and only if 
\[\mptncl(G) +  \{K\trans\!A + AK: K\in\Rskewn\} + \vspan\{A^k: k = 0,\ldots, q-1\} = \Rsn.\]
\item[{\rm (c)}] The matrix $A$ has the SAP if and only if 
\[\mptncl(G) +  \{L\trans\!A + AL: L\in\Rmn\} = \Rsn.\]
\end{enumerate}
\end{proposition}

\section{Inverse function theorem}
\label{sec:invft}

In this section, we revisit the classical inverse function theorem and construct a version that is suitable for our purpose.  

Let $F(\bu)$ be a function from an open subset of a vector space $U$ to another vector space $W$.
Recall that a function is \emph{smooth} if it is infinitely differentiable throughout its domain.  
Let $\bu_0$ be a point in the domain $F(\bu_0) = \bw_0$.  Note that the tangent space around $\bu_0$ in $U$ and the tangent space around $\bw_0$ in $W$ are $U$ and $W$, respectively.  The \emph{derivative} of $F$ at $\bu_0$, denoted as $\dot{F}\Big|_{\bu = \bu_0}$, is a linear operator 
\[\bd \mapsto \lim_{t\rightarrow 0}\frac{F(\bu_0 + t\bd) - F(\bu_0)}{t}\]
from the tangent space of $U$ to the tangent space of $W$ that takes a direction $\bd$ and outputs the directional derivative.  When both $U$ and $W$ are finite-dimensional, it is easy to check if $\dot{F}\Big|_{\bu = \bu_0}$ is injective, surjective, or invertible since it is linear.  In particular, when $\dim U = \dim  W < \infty$, we know one of the three properties is enough to imply the other two.

\begin{example}
Let $U = \Rskewn$ and $W = \Rmn$.  Then the function $F(K) = e^{K}$ from an open subset of $U$ to $W$ has 
\[\dot{F}\Big|_{K = O} = K,\]
since 
\[\begin{aligned}
\lim_{t\rightarrow 0}\frac{e^{O+Kt} - e^{O}}{t} &= 
\lim_{t\rightarrow 0}\frac{1}{t}\left[\frac{(Kt)^0}{0!} + \frac{(Kt)^1}{1!} + \frac{(Kt)^2}{2!} + \frac{(Kt)^3}{3!} + \cdots - I\right] \\
 &= 
\lim_{t\rightarrow 0}\left[\frac{K^1}{1!} + \frac{K^2t^1}{2!} + \frac{K^3t^2}{3!} + \cdots \right] = K.\\
\end{aligned}\]
\end{example}

\begin{example}
Let $U = \calB_n$ and $W = \Rmn$.  Let $F: U\rightarrow W$ be a function defined by $F(L) = (I + L)^{-1}$.  Then 
\[\dot{F}\Big|_{L = O} = -L\]
since 
\[\begin{aligned}
\lim_{t\rightarrow 0}\frac{(I + Lt)^{-1} - I}{t} &= 
\lim_{t\rightarrow 0}\frac{1}{t}\left[I - Lt + (Lt)^2 - (Lt)^3 + \cdots - I\right] \\
 &= 
\lim_{t\rightarrow 0}\left[- L + L^2 t - L^3t^2 + \cdots \right] = -L.\\
\end{aligned}\]
\end{example}

Theorem~\ref{thm:invft} is one of the common statements of the inverse function theorem, which can be found in, e.g., \cite[Theorem~1.12]{GMM13}. 

\begin{theorem}[Inverse function theorem]
\label{thm:invft}
Let $U$ and $W$ be finite-dimensional vector spaces over $\mathbb{R}$ with $\dim U = \dim W$.  
Let $F$ be a smooth function from an open subset of $U$ to $W$ with $F(\bu_0) = \bw_0$.  
If $\dot{F}\Big|_{\bu = \bu_0}$ is invertible, then there are an open subset $W'\subseteq W$ containing $\bw_0$ and a smooth function $T:W'\rightarrow U$ such that $T(\bw_0) = \bu_0$ and $F\circ T$ is the identity map on $W'$.  
\end{theorem}

The condition $\dim U = \dim W$ is necessary so that the derivative can possibly be invertible.  However, the inverse function theorem still works when $U$ and $W$ have different dimensions and the derivative is surjective.  Here we state another version of the inverse function theorem and include the proof for completeness.

\begin{theorem}
\label{thm:invftsur}
Let $U$ and $W$ be finite-dimensional vector spaces over $\mathbb{R}$.  
Let $F$ be a smooth function from an open subset of $U$ to $W$ with $F(\bu_0) = \bw_0$.  
If $\dot{F}\Big|_{\bu = \bu_0}$ is surjective, then there are an open subset $W'\subseteq W$ containing $\bw_0$ and a smooth function $T:W'\rightarrow U$ such that $T(\bw_0) = \bu_0$ and $F\circ T$ is the identity map on $W'$. 
\end{theorem}
\begin{proof}
Let $\dot{F} = \dot{F}\Big|_{\bu = \bu_0}$.  Let $U'$ be the orthogonal complement of the kernel $\ker(\dot{F})$.  Since $\dot{F}$ is surjective, the restriction of $\dot{F}$ on $U'$ is an invertible operator.  Let $\bu_0 + U'$ be the affine subspace 
\[\{\bu_0 + \bu: \bu \in U'\}.\]
Let $\tilde{F}$ be the restriction of $F$ on the intersection of $U'$ and the domain of $F$.  Thus, $\tilde{F}$ is a function whose domain and codomain are vector spaces of the same dimension, and $\dot{\tilde{F}}\Big|_{\bu = \bu_0}$
equals the restriction of $\dot{F}$ on $U'$, which is invertible.  By Theorem~\ref{thm:invft}, there is an open subset $W'\subseteq W$ containing $\bw_0$ and a smooth function $T:W'\rightarrow \bu_0 + U' \subseteq U$ such that $T(\bw_0) = \bu_0$ and $\tilde{F}\circ T$ is the identity map on $W'$.  Since $\tilde{F}$ is a restriction of $F$, naturally $F\circ T$ is also the identity map on $W'$.
\end{proof}

\section{Bifurcation lemma}
\label{sec:bifur}

We establish the so-called bifurcation lemma in this section.  Before this verification, we observe that Proposition~\ref{prop:sxpsurj} demonstrates that each strong property is equivalent to the derivative of a certain perturbation function being surjective.  


\begin{proposition}
\label{prop:sxpsurj}
Let $G$ be a graph on $n$ vertices and $A\in\mptn(G)$ with $q$ distinct eigenvalues.
\begin{enumerate}
\item[{\rm (a)}] For the function $F:\mptncl(G)\times\Rskewn\rightarrow\Rsn$ defined by 
\[F(B,K) = e^{-K}(A + B)e^K,\]
the derivative $\dot{F}\Big|_{\substack{B = O\\K = O}}$ is surjective if and only if $A$ has the SSP.
\item[{\rm (b)}] For the function $F:\mptncl(G)\times\Rskewn\times\mathbb{R}^q\rightarrow\Rsn$ defined by 
\[F(B,K,\bc) = e^{-K}p(A + B)e^K,\]
\[
\text{where }\bc = (c_0,\ldots, c_{q-1})\text{ and }p(x) = x + \sum_{k=0}^{q-1}c_kx^k,
\]
the derivative $\dot{F}\Big|_{\substack{B = O\\K = O\\\bc = \bzero\\}}$ is surjective if and only if $A$ has the SMP.

\item[{\rm (c)}] For the function $F:\mptncl(G)\times\calB_n\rightarrow\Rsn$ defined by 
\[F(B,L) = (I + L)\trans(A + B)(I + L),\]
the derivative $\dot{F}\Big|_{\substack{B = O\\L = O\\}}$ is surjective if and only if $A$ has the SAP.
\end{enumerate}
\end{proposition}
\begin{proof}
We show that the partial derivatives of each function correspond to the spaces related to each of the strong properties.

\noindent\textbf{Case (a).}  Let $\dot{F} = \dot{F}\Big|_{\substack{B = O\\K = O}}$.  Let $F_B = \frac{\partial F}{\partial B}\Big|_{\substack{B = O\\K = O}}$ and $F_K = \frac{\partial F}{\partial K}\Big|_{\substack{B = O\\K = O}}$ be the corresponding partial derivatives.  Thus, 
\[\range(\dot{F}) = \range(F_B) + \range(F_K)\]
and $\dot{F}$ is surjective if and only if $\range(F_B) + \range(F_K) = \Rsn$.  Observe that $\range(F_B) = \mptncl(G)$ and 
\[\range(F_K) = \{K\trans\!A + AK: K \in \Rskewn\}.\]
Thus, $\range(F_B) + \range(F_K) = \Rsn$ if and only if $A$ has the SSP by Proposition~\ref{prop:sxptan}.

\noindent\textbf{Case (b).}  Similarly, let $\dot{F} = \dot{F}\Big|_{\substack{B = O\\K = O\\\bc=\bzero}}$ and let $F_B$, $F_K$, and $F_{\bc}$ be the corresponding partial derivatives at $B = O$, $K = O$, and $\bc = \bzero$.  By direct computation, we have $\range(F_B) = \mptncl(G)$, 
\[\range(F_K) = \{K\trans\!A + AK: K \in \Rskewn\},\]
and 
\[\range(K_{\bc}) = \vspan\{A^k: k = 0,\ldots, q-1\}.\]
By Proposition~\ref{prop:sxptan}, $\dot{F}$ is surjective if and only if $A$ has the SMP.

\noindent\textbf{Case (c).}  Following the same workflow and with $F_B$ and $F_L$ being the corresponding partial derivatives at $B = O$ and $L = O$, we arrive at the conclusion by observing $\range(F_B) = \mptncl(G)$ and 
\[\range(F_L) = \{L\trans\!A + AL : L \in \Rmn\}.\]
\end{proof}

Note that the functions in Proposition~\ref{prop:sxpsurj} are all compositions of perturbations.  The function $B\mapsto A+B$ is a perturbation of $A$ that does not change the pattern (when $B$ is small enough).  The function $K\mapsto e^{-K}Me^K$ does not change the spectrum of $M$, the function $(K,\bc)\mapsto e^{-K}p(M)e^K$ does not change the ordered multiplicity list, and the function $L\mapsto (I+L)\trans M(I+L)$ preserves the inertia (and rank) of $L$.

\begin{theorem}[Bifurcation lemma]
\label{thm:bifur}
Let $G$ be a graph on $n$ vertices and $A\in\mptn(G)$.  If $A$ has the 
\begin{enumerate}
\item[\rm (a)] SSP,
\item[\rm (b)] SMP, or 
\item[\rm (c)] SAP,
\end{enumerate}
respectively, then there is $\epsilon > 0$ such that for any $M\in\Rsn$ with $\|M - A\| < \epsilon$, a matrix $A'\in\mptn(G)$ exists such that 
\begin{enumerate}
\item[\rm (a)] $\spec(A')=\spec(M)$ and $A'$ has the SSP, 
\item[\rm (b)] $\oml(A')=\oml(M)$ and $A'$ has the SMP, or
\item[\rm (c)] $\piner(A')=\piner(M)$ and $A'$ has the SAP, 
\end{enumerate}
respectively.
\end{theorem}
\begin{proof}
We prove the case for the SSP only, since the other cases are similar.  Let $F$ be the function in Proposition~\ref{prop:sxpsurj}(a) and $\dot{F}$ its derivative at zero.  Since $A$ has the SSP, $\dot{F}$ is surjective.  By the inverse function theorem (Theorem~\ref{thm:invftsur}), there are an open subset $W'\subseteq\Rsn$ containing $A$ and a smooth function $T$ such that $T(A) = (O,O)$ and $F\circ T$ is the identity map on $W'$.  Thus, for any $M\in W'$, we may let $(B',K') = T(M)$.  Then
\[F(B', K') = e^{-K'}(A + B')e^{K'} = F(T(M)) = M,\]
which implies $A+B'$ and $M$ have the same spectrum since $e^{K'}$ is an orthogonal matrix and $e^{-K'}$ is its inverse.
Therefore, when $\|M - A\|$ is sufficiently small, $A' = A + B'$ is a matrix in $\mptn(G)$ and $\spec(A') = \spec(M)$.  Moreover, the derivative of $F$ at $B = B'$ and $K = K'$ remains surjective when the perturbation is small enough, so $A'$ has the SSP as well.  
The other cases can be obtained by choosing an appropriate function in Proposition~\ref{prop:sxpsurj}.
\end{proof}

Let $A$ be a real symmetric matrix and suppose $QDQ\trans$ is its diagonalization with $Q$ an orthogonal matrix.  Since the Frobenius norm is unitarily invariant,  
\[
\|QD'Q\trans - QDQ\trans\|^2 =  \|D' - D\|^2. \]
Thus, $M$ can be chosen as $QD'Q\trans$ for any $D'$ with $\|D' - D\| < \epsilon$.  
By applying the bifurcation lemma to a matrix $M$ with the nearby spectrum, ordered multiplicity list, or inertia, the following corollaries are immediate.

\begin{corollary}
\label{cor:bifurssp}
Let $G$ be a graph and $A\in\mptn(G)$ a matrix with the SSP and $\spec(A) = \{\lambda_1,\ldots,\lambda_n\}$.  Then there is $\epsilon > 0 $ such that for any $\mu_1,\ldots,\mu_n$ with $|\mu_i - \lambda_i| < \epsilon$ for all $i = 1,\ldots,n$, a matrix $A'\in\mptn(G)$ exists with the SSP and $\spec(A') = \{\mu_1,\ldots,\mu_n\}$. 
\end{corollary}

\begin{corollary}
\label{cor:bifursmp}
Let $G$ be a graph and $A\in\mptn(G)$ a matrix with the SMP and $\oml(A) = \oml$.  Then for any refinement $\oml'$ of $\oml$, there is a matrix $A'\in\mptn(G)$ with the SMP and $\oml(A') = \oml'$. 
\end{corollary}

Corollary~\ref{cor:bifursappin} is known as the stable northeast lemma for graphs, which was developed in \cite[Corollary~2]{AHLvdH13}.  It is also an immediate consequence of the bifurcation lemma.

\begin{corollary}
\label{cor:bifursappin}
Let $G$ be a graph on $n$ vertices and $A\in\mptn(G)$ a matrix with the SAP such that $\piner(A) = (a,b)$ with $a + b < n$.  Then each of the partial inertia $(a+1,b)$ and $(a,b+1)$ can be achieved by a matrix $A'\in\mptn(G)$ with the SAP. 
\end{corollary}

Note that Corollary~\ref{cor:bifursappin} remains true if the assumption of SAP is dropped, but the resulting matrix $A'$ is not guaranteed to also have the SAP.  See \cite[Lemma~1.1]{BHL09}.

\begin{corollary}
\label{cor:bifursaprank}
Let $G$ be a graph on $n$ vertices and $A\in\mptn(G)$ a matrix with the SAP and $\rank(A) = r$.  Then for any $r'$ with $r\leq r'\leq n$ there is a matrix $A'\in\mptn(G)$ with the SAP and $\rank(A') = r'$.
\end{corollary}

Similar to the comments after Corollary~\ref{cor:bifursappin}, the SAP in Corollary~\ref{cor:bifursaprank} can be removed.  That is, if $A\in\mptn(G)$ has rank $r$, then for any $r'$ with $r\leq r'\leq n$ there is a matrix $A'$ with $\rank(A') = r'$.  This is well-known by the following argument:  Pick an invertible matrix $B\in\mptn(G)$; then since $B$ can be obtained from $A$ by a sequence of rank-$1$ perturbations without leaving $\mptn(G)$, the ranks of matrices on the way from $A$ to $B$ include at least all the integers between $r$ to $n$.  The matrices resulting from this process are not, however, guaranteed to have the SAP.

\section{Applications}
\label{sec:app}

In this section we provide several consequences of the bifurcation lemma.

\subsection{Minimum number of distinct eigenvalues}

Recall that $q(A)$ is the number of distinct eigenvalues of $A$.  With the bifurcation lemma, each eigenvalue of high multiplicity can be divided into a pair of eigenvalues each with smaller multiplicity.  Therefore, one may increase the number of distinct eigenvalues by one without changing the pattern of the matrix.  

\begin{theorem}
\label{thm:qincrement}
Let $G$ be a graph on $n$ vertices and $A\in\mptn(G)$ a matrix with the SSP or the SMP, respectively.  Then for any $q'$ with $q(A) \leq q'\leq n$, there is a matrix $A'\in\mptn(G)$ with the SSP or the SMP, respectively,  such that $q(A') = q'$.
\end{theorem}

It remains an open question whether the assumption of the SSP or the SMP in Theorem~\ref{thm:qincrement} can be dropped.

\subsection{Cycles}

Let $C_n$ be the cycle on $n$ vertices.  It is known \cite{Ferguson80} that a set of real numbers $\lambda_1\leq \cdots \leq \lambda_n$ is the spectrum of a matrix in $\mptn(C_n)$ if and only if 
\[\lambda_1 \leq \lambda_2 < \lambda_3 \leq \lambda_4 < \lambda_5 \leq \cdots \]
or 
\[\lambda_1 < \lambda_2 \leq \lambda_3 < \lambda_4 \leq \lambda_5 < \cdots .\]
In other words, the ordered multiplicity lists realizable by matrices in $\mptn(G)$ are the refinements of $(2,\ldots,2)$ or $(1,2,\ldots,2,1)$ when $n$ is even and are the refinements of $(2,\ldots,2,1)$ or $(1,2,\ldots,2)$ when $n$ is odd. 
We will see that these four maximal ordered multiplicity lists can be achieved with the SMP, so we have the following theorem.

\begin{theorem}
\label{thm:cyclesmp}
Every ordered multiplicity list that can be achieved by a matrix in $\mptn(C_n)$ can be achieved by a matrix in $\mptn(C_n)$ with the SMP.
\end{theorem}
\begin{proof}
The IEP-$G$ has been solved for small graphs up to $5$ vertices; see \cite{IEPG2} and the references therein.  So we know $(2,1)$ and $(1,2)$ can be realized by matrices in $\mptn(C_3)$ with the SSP (so the SMP); we also know $(2,2)$ and $(1,2,1)$ can be realized by matrices in $\mptn(C_4)$ with the SSP (so the SMP).  

Let $k\geq 2$.  By \cite[Theorem~48]{gSAP}, for any even cycle $C_{2k}$ there is a matrix $A\in\mptn(C_{2k})$ with the SMP and $\oml(A) = (2,\ldots,2)$.  By the decontraction lemma \cite[Lemma~6.12]{IEPG2}, for any odd cycle $C_{2k+1}$ there are matrices $A_1'$ and $A_2'$ in $\mptn(C_{2k+1})$ with the SMP and $\oml(A_1') = (1,2,\ldots,2)$ and $\oml(A_2') = (2,\ldots,2,1)$.  By applying the decontraction lemma again to $A_1'$ or $A_2'$, for any even cycle $C_{2k+2}$, there is a matrix $A''\in\mptn(C_{2k+2})$ with the SMP and $\oml(A) = (1,2,\ldots,2,1)$.  Combining all these results, we know for $n\geq 5$, there are matrices with the SMP in $\mptn(C_n)$ realizing the ordered multiplicity lists $(2,\ldots,2)$ or $(1,2,\ldots,2,1)$ if $n$ is even and the lists $(1,2,\ldots,2)$ and $(2,\ldots,2,1)$ if $n$ is odd.

By the bifurcation lemma (Corollary~\ref{cor:bifursmp}), the desired result holds. 
\end{proof}


It is not yet known, however, whether every realizable spectrum of matrices in $\mptn(C_n)$ can be realized by a matrix in $\mptn(C_n)$ with the SSP.

\section{Not necessarily symmetric matrices}
\label{sec:nsym}

The idea of bifurcation also applies to matrices that are not necessarily symmetric and has applications to the study of sign patterns.  A \emph{sign pattern} $\sptn$ is an array whose entries are $+$, $-$, or $0$, and the \emph{qualitative class} $\qual(\sptn)$ is the set of all matrices of the same dimensions as $\sptn$ and the signs of whose entries are prescribed by $\sptn$.  We say $\sptn'$ is a \emph{superpattern} of $\sptn$ if $\sptn'$ can be obtained from $\sptn$ by replacing some $0$'s with $+$'s or $-$'s.  (Note that $\sptn$ is also a superpattern of itself.)

The terminology was introduced in \cite{DJOvdD00} that an $n\times n$ sign pattern $P$ is said to be a \emph{spectrally arbitrary pattern} if for each monic real-coefficient polynomial $p(x)$ of degree $n$, there is a matrix $A\in\qual(\sptn)$ whose characteristic polynomial is $p(x)$, or in other words that every set of $n$ complex numbers that is invariant under conjugation can be realized by a matrix $A\in\qual(\sptn)$.  

Let $\sptn$ be an $n\times n$ sign pattern and $A\in\qual(\sptn)$.  The nilpotent-centralizer method \cite{GB13} states that if 
\begin{itemize}
\item $A$ is a nilpotent matrix of index $n$ and 
\item $X = O$ is the only matrix satisfying $A\circ X = O$ and $[A,X\trans] = O$,
\end{itemize}
then any of its superpatterns is a spectrally arbitrary pattern.  In the following, we will build a more flexible version of this theorem and see that it follows the same flavor of the bifurcation lemma.

Let $\sptn = \begin{bmatrix} p_{ij} \end{bmatrix}$ be an $n\times n$ sign pattern.  
We define $\qualv(\sptn)$ as the set of matrices $A = \begin{bmatrix} a_{ij} \end{bmatrix}\in\Rmn$ such that $a_{ij} = 0$ if $p_{ij} = 0$ (and $a_{ij}$ can be any real number if $p_{ij}$ is $+$ or $-$).  The set $\qualv(\sptn)$ is a vector space; indeed, it is the tangent space of $\qual(\sptn)$ at any point.

\begin{definition}
Let $A$ be an $n\times n$ matrix.  Then $A$ has the \emph{non-symmetric strong spectral property} (nSSP) if $X = O$ is the only matrix satisfying $A\circ X = O$ and $[A,X\trans] = O$.
\end{definition}

\begin{remark}
The nSSP seemingly misses one condition $I\circ X = O$ comparing to the SSP.  However, both $A\circ X = I\circ X = O$ in the SSP and $A\circ X$ in the nSSP can be interpred as ``$X$ has to be zero on those potentially nonzero entries.''  From this point of view, the nSSP is a natural analogue of the SSP for non-symmetric matrices.
\end{remark}

\begin{theorem}[Bifurcation lemma for non-symmetric matrices]
\label{thm:bifurnsym}
Let $\sptn$ be an $n\times n$ sign pattern and  $A\in\qual(\sptn)$.  If $A$ has the nSSP, then there is $\epsilon > 0$ such that for any $M\in\Rmn$ with $\|M - A\| < \epsilon$, a matrix $A'\in\qual(\sptn)$ exists such that $A'$ and $M$ are similar and $A'$ has the nSSP. 
\end{theorem}
\begin{proof}

Given a matrix $A\in\qual(\sptn)$, consider the function $F:\qualv(\sptn)\times \calB_n\rightarrow \Rmn$ defined by 
\[F(B, L) = (I + L)^{-1}(A + B)(I + L).\]
Let $\dot{F} = \dot{F}\Big|_{\substack{B = O\\L = O}}$.  By direct computation, the range of $\dot{F}$ is 
\[\qualv(\sptn) + \{-LA + AL : L\in\Rmn\}.\]
Note that in the vector space $\Rmn$, the orthogonal complement of $\qualv(\sptn)$ is \[\{X\in\Rmn: A\circ X = O\},\]
and the orthogonal complement of $\{-LA + AL : L\in\Rmn\}$ is 
\[\{X\in\Rmn: [A, X\trans] = O\}\]
since 

\[\begin{aligned}
0 = \inp{-LA + AL}{X} &= \tr(X\trans(-LA + AL)) \\
&= \tr(-X\trans LA + X\trans\!AL) \\ 
&= \tr(-L AX\trans + L X\trans\!A) \\
&= -\tr(L(AX\trans - X\trans\!A)) \\
&= -\inp{AX\trans - X\trans\!A}{L\trans},
\end{aligned}\]
for $L\in\Rmn$.
Therefore, $A$ has the nSSP if and only if $\dot{F}$ is surjective.  

By the inverse function theorem (Theorem~\ref{thm:invftsur}), there is $\epsilon > 0$ such that for any $M\in\Rmn$ with $\|M - A\| < \epsilon$, matrices $B'\in\qualv(\sptn)$ and $L'\in\Rmn$ exist with
\[F(B',L') = (I + L')^{-1}(A + B')(I + L') = M,\]
which means $A' = A + B'$ is similar to $M$.  When $\|M - A\|$ is sufficiently small, $\|B'\|$ is small and $A'\in\qual(\sptn)$.  Moreover, when the perturbation is sufficiently small, the derivative of $F$ at $B = B'$ and $L = L'$ remains surjective, so $A'$ has the nSSP.
\end{proof}

The function $F$ used in the bifurcation lemma (Theorem~\ref{thm:bifurnsym}) is a composition of $M\mapsto (I + L)^{-1}M(I + L)$ and $M\mapsto M + B$.  In the following, we will see by changing the order of the composition, we will get the superpattern lemma.  In short, the bifurcation lemma allows us to find a matrix with the same pattern but with a slightly different Jordan canonical form, while the superpattern allows us to find a matrix with the same Jordan canonical form but with a slightly different pattern.  Note that the superpattern lemma is a non-symmetric analogue of the supergraph lemma \cite[Theorem~10]{gSAP}.

\begin{theorem}[Superpattern lemma]
\label{thm:super}
Let $\sptn$ be an $n\times n$ sign pattern and  $A\in\qual(\sptn)$.  If $A$ has the nSSP, then for any superpattern $\sptn'$ of $\sptn$, a matrix $A'\in\qual(\sptn')$ exists such that $A'$ and $A$ are similar and $A'$ has the nSSP. 
\end{theorem}
\begin{proof}
Given a matrix $A\in\qual(\sptn)$, consider the function $F:\qualv(\sptn)\times \calB_n\rightarrow \Rmn$ defined by 
\[F(B, L) = (I + L)^{-1}A(I + L)  + B.\]
Let $\dot{F} = \dot{F}\Big|_{\substack{B = O\\L = O}}$.  By the arguments in the proof of Theorem~\ref{thm:bifurnsym}, $A$ has the nSSP if and only if $\dot{F}$ is surjective.

By the inverse function theorem (Theorem~\ref{thm:invftsur}), there is $\epsilon > 0$ such that for any $M\in\Rmn$ with $\|M - A\| < \epsilon$, matrices $B'\in\qualv(\sptn)$ and $L'\in\Rmn$ with
\[F(B',L') = (I + L')^{-1}A(I + L') + B' = M\]
exist, which means $A' = M - B'$ is similar to $A$.  We may choose $M$ as the matrix obtained from $A$ by adding $+s$ or $-s$ at the extra nonzero entries in $\sptn'$ but not in $\sptn$.  The quantity $s$ can be chosen sufficiently small so that $\|M - A\| < \epsilon$.  Since $B'$ is a small perturbation on the nonzero entries of $A$, $A' = M - B'$ is a matrix in $\qual(\sptn')$.  When the perturbation is sufficiently small, the derivative of $F$ at $B = B'$ and $L = L'$ remains surjective, so $A'$ has the nSSP.
\end{proof}

With the bifurcation lemma and the superpattern lemma, we are ready to prove Corollary~\ref{cor:fNC}, a more flexible version of the nilpotent-centralizer method.  Note that in Corollary~\ref{cor:fNC}, the nilpotent matrix is no longer required to have index $n$.  Instead, all that is required is that the spectra within some neighborhood of a nilpotent matrix include all sufficiently small spectra that are invariant under conjugation. The following known result will allow us to show that every nilpotent matrix meets this requirement.

\begin{theorem}
\label{thm:realform}
{\rm \cite[Theorem~2.3.4]{HJ85}}
If $A \in \Rmn$, there is an orthogonal matrix $Q \in \Rmn$ such that 
\[Q\trans\!A Q = 
\begin{bmatrix}
 A_1 & ~ & ~ & ~ \\
 ~ & A_2 & \multicolumn{2}{c}{\smash{\raisebox{.5\normalbaselineskip}{\Large $*$}}} \\
 ~ & ~ & \ddots & ~ \\
 \multicolumn{2}{c}{\smash{\raisebox{0.5\normalbaselineskip}{\Large $O$}}} & ~ & A_k
\end{bmatrix} 
\in \Rmn,
\]
where $1\leq k\leq n$ and each $A_i$ is a real $1\times 1$ matrix, or a real $2\times 2$ matrix with a non-real pair of complex conjugate eigenvalues.  The diagonal blocks $A_i$ may be arranged in any prescribed order. 
\end{theorem}

\begin{corollary}
\label{cor:nearbynil}
Let $A\in\Rmn$ be nilpotent and let $\epsilon > 0$ be given.
Then for any complex multiset $\Lambda = \{\lambda_1,\ldots,\lambda_n\}$, invariant under conjugation, with 
\[\sum_{i=1}^n|\lambda_i|^2 <  \epsilon^2,\]
a matrix $M$ exists with $\|M - A\| < \epsilon$ and $\spec(M) = \Lambda$.  
\end{corollary}
\begin{proof}
Since $\|M - A\| = \|Q\trans M Q - Q\trans A Q\|$ for any orthogonal matrix $Q$,
by Theorem~\ref{thm:realform} we may assume without loss of generality that $A$ is strictly upper triangular, with all diagonal entries equal to $0$.
For each real eigenvalue $\lambda_i$ the perturbation of a $1 \times 1$ diagonal block
from
$[0] \mbox{ in } A \mbox{ to } [\lambda_i] \mbox{ in } M$
contributes
exactly $|\lambda_i|^2$ to $\|M - A\|^2$.
For each non-real conjugate pair $\{\lambda_i, \lambda_j\}$ whose real part is $a$,
and the absolute value of whose imaginary part is $b > 0$,
the perturbation of some $2 \times 2$ diagonal block
\[
\mbox{ from }
\begin{bmatrix}
0 & x \\
0 & 0
\end{bmatrix} \mbox{ in } A\ \ 
\begin{cases}
  \mbox{ to }
  \begin{bmatrix}
    a & -b \\
    b & a
  \end{bmatrix} \mbox{ in } M, & -b \le x < 0 \\
  ~&~ \\
  \mbox{ to }
  \begin{bmatrix}
    a & b \\
    -b & a
  \end{bmatrix} \mbox{ in } M, & 0 \le x \le b \\
  ~&~ \\
  \mbox{ to }
  \begin{bmatrix}
    a & x \\
    \frac{-b^2}{x} & a
  \end{bmatrix} \mbox{ in } M, & |x| > b
\end{cases}
\]
contributes no more than $|\lambda_i|^2 + |\lambda_j|^2 = 2a^2 + 2b^2$
to $\|M - A\|^2$ in every case, regardless of the value of $x$.
It follows that $\|M - A\|^2 \le \epsilon^2$ holds, and thus $\|M - A\| < \epsilon$,
for the perturbed matrix $M$ with spectrum $\Lambda$.
\end{proof}




\begin{corollary}
\label{cor:fNC}
Let $\sptn$ be an $n\times n$ sign pattern and $A\in\qual(\sptn)$.  If $A$ is nilpotent and has the nSSP, then any superpattern of $\sptn$ is a spectrally arbitrary pattern.
\end{corollary}
\begin{proof}
By the bifurcation lemma for non-symmetric matrices (Theorem~\ref{thm:bifurnsym}), there is $\epsilon > 0$ such that for any $M$ with $\|M - A\| < \epsilon$, a matrix $A'\in\qual(\sptn)$ exists such that $A'$ and $M$ are similar and $A'$ has the nSSP.    



By Corollary~\ref{cor:nearbynil}, for any set of complex numbers $\Lambda = \{\lambda_1,\ldots,\lambda_n\}$, invariant under conjugation, with 
\[\sum_{i=1}^n|\lambda_i|^2 <  \epsilon^2,\]
a matrix $M$ exists with $\|M - A\| < \epsilon$ and $\spec(M) = \Lambda$.  

Combining the above results, $\sptn$ allows any sufficiently small spectrum that is invariant under conjugation.  By scaling $A'$ to $kA'$, $\sptn$ allows any spectrum that is invariant under conjugation.  Moreover, by the superpattern lemma (Theorem~\ref{thm:super}), any superpattern of $\sptn$ is also a spectrally arbitrary pattern since $\sptn$ is a spectrally arbitrary pattern and the matrices $A$, $A'$, and $kA'$ have the nSSP.
\end{proof}

This leads to a theorem developed by Pereira~\cite{Pereira07}.  Recall that a sign pattern is \emph{full} if each of its entries is nonzero.  Also, every matrix in the qualitative class of a full sign pattern naturally has the nSSP by definition. 

\begin{corollary}
\label{cor:Pereira07}
{\rm \cite[Theorem~1.2]{Pereira07}}
If $\sptn$ is a full $n\times n$ sign pattern and $\qual(\sptn)$ contains a nilpotent matrix, then $\sptn$ is a spectrally arbitrary pattern.
\end{corollary}

\begin{remark}
Comparing to the nilpotent-centralizer method, Corollary~\ref{cor:fNC} provides a more flexible condition for finding a spectrally arbitrary pattern since the nilpotent matrix $A$ is not necessarily of index $n$.  However, as we can see later in Corollary~\ref{cor:indexn}, a sign pattern $\sptn$ can be verified as a spectrally arbitrary pattern by the nilpotent-centralizer method if and only if $\sptn$ can be verified as a spectrally arbitrary pattern by Corollary~\ref{cor:fNC}.
\end{remark}

\begin{theorem}
\label{thm:realJordan}
{\rm \cite[Theorem~3.1.11]{HJ85}}
Let $A\in\Rmn$ be a matrix with only real eigenvalues.  Then there is a nonsingular matrix $Q\in\Rmn$ such that $A = QJQ^{-1}$, where $J$ is the Jordan canonical form of $A$.
\end{theorem}

\begin{corollary}
\label{cor:indexn}
Let $\sptn$ be a sign pattern.  If there is a nilpotent matrix $A\in\qual(\sptn)$ with the nSSP, then there is a nilpotent matrix $A'\in\qual(\sptn)$ of index $n$ and with the nSSP.
\end{corollary}
\begin{proof}
According to Theorem~\ref{thm:realJordan}, every $n\times n$ nilpotent matrix has a nilpotent matrix of index $n$ that is arbitrarily close to it.  Thus, the desired result follows from the bifurcation lemma for non-symmetric matrices (Theorem~\ref{thm:bifurnsym}).
\end{proof}

The nilpotent-Jacobian method introduced in \cite{DJOvdD00} is another method for verifying a spectrally arbitrary pattern, and it was developed earlier than the nilpotent-centralizer method.  Roughly speaking, the nilpotent-Jacobian method considers a map $f$ from $n$ nonzero entries on an $n\times n$ nilpotent matrix $A\in\qual(\sptn)$ to the $n$ non-leading coefficients of the characteristic polynomial $\det(xI - A)$, and it states that if the derivative of $f$ is surjective (equivalently, invertible) then any superpattern of $\sptn$ is spectrally arbitrary.  On  the one hand, for the nilpotent-centralizer method, it was shown that if $A$ is a nilpotent matrix of index $n$ and $A$ has the nSSP, then the map $f$ with the desired properties for the nilpotent-Jacobian method exists \cite{GB13}.  On the other hand, in \cite{BvMvT11}, it was shown that such a map $f$ exists only when the nilpotent matrix $A$ has index $n$.  In summary, when $A$ is a nilpotent matrix, the map $f$ for the nilpotent-Jacobian method exists if and only if $A$ is of index $n$ and $A$ has the nSSP.  From this point of view, the nilpotent-Jacobian method and the nilpotent-centralizer method are somehow equivalent, while the later one avoid the hassle of finding the appropriate $n$ entries.  Example~\ref{ex:fNC} provides some reasoning of why the previous two methods require the nilpotency of index $n$ while Corollary~\ref{cor:fNC} only needs the nilpotency.

\begin{example}
\label{ex:fNC}
Let 
\[A = \begin{bmatrix}
 -1 & 1 & -1 \\
 -2 & 2 & -2 \\ 
 -1 & 1 & -1 \\ 
\end{bmatrix} \text{ and }
B = \begin{bmatrix}
x_{11} & x_{12} & x_{13} \\
x_{21} & x_{22} & x_{23} \\
x_{31} & x_{32} & x_{33} \\
\end{bmatrix}.\]
Then $A$ is a nilpotent matrix of index $2$, which is smaller than the size of $A$.  Let $\det(xI - (A + B)) = c_0 + c_1x + c_2x^2 + x^3$ such that $c_0,c_1,c_2$ are polynomials in variables in $B$.  Let $f$ be the map $B\mapsto (c_0,c_1,c_2)$.  In particular, 
\[c_0 = \det(-(A + B)) = (-1)^3\det(A + B).\] 
Notice that at $B=O$, no single variable in $B$ is able to change the value of $\det(A + B)$ since $\rank(A) = 1$ and all $2\times 2$ submatrices are singular.  Therefore, the derivative of $c_0$ at $B = O$ with respect to any $x_{ij}$ is $0$, and the derivative of $f$ is not surjective.  Consequently, the nilpotent-Jacobian method and the nilpotent-centralizer method do not apply in this case.  

In contrast, the variables in $B$ naturally perturb $A$ into any matrices nearby $A$, so Corollary~\ref{cor:fNC} still applies.
\end{example}

The bifurcation lemma also has application to the possible inertia sets of a sign pattern.  Let $A$ be an $n\times n$ matrix.  The \emph{inertia} of a not necessarily symmetric $A$ is 
\[\iner(A) = (n_+(A), n_-(A), n_0(A)),\]
where $n_+(A)$, $n_-(A)$, and $n_0(A)$ are the number of eigenvalues with positive, negative, and zero real part.  
The \emph{refined inertia} was considered in \cite{KOSvdDvdHvM09} and is defined as 
\[\riner(A) = (n_+(A), n_-(A), n_z(A), 2n_p(A)),\]
where $n_+(A)$ and $n_-(A)$ are the defined as those in the inertia while $n_z(A)$ and $2n_p(A)$ are the number of zero eigenvalues and the number of nonzero pure-imaginary eigenvalues, respectively.  Also, an $n\times n$ sign pattern is called an \emph{inertially arbitrary pattern} if it can realize all possible inertia $(n_+, n_-, n_0)$ with $n_+ + n_- + n_0 = n$.

\begin{corollary}
\label{cor:finer}
Let $\sptn$ be an $n\times n$ sign pattern and $A\in\qual(\sptn)$.  If $n_0(A) = n$, $n_z(A) \geq 2$, and $A$ has the nSSP, then $\sptn$ is an inertially arbitrary pattern.
\end{corollary}
\begin{proof}
According to the bifurcation lemma for non-symmetric matrices (Theorem~\ref{thm:bifurnsym}), a zero eigenvalue can be perturbed into any nearby real eigenvalue, while a conjugate pair of nonzero pure-imaginary eigenvalues can be perturbed into two eigenvalues whose real parts are of the same sign.  

Suppose $(p,q,n-p-q)$ is the target inertia.  If $2n_p(A) \geq p$, then we perturb $\left\lfloor\frac{p}{2}\right\rfloor$ conjugate pairs of nonzero pure-imaginary eigenvalues to have positive real parts, and, if in addition $p$ is odd, perturb a zero eigenvalue be a positive eigenvalue.  If not, we perturb all $n_p(A)$ conjugate pairs of nonzero pure-imaginary eigenvalues to have positive real parts, and perturb $p - 2n_p(A)$ zero eigenvalues to be positive.  After doing similar perturbations to make $q$ eigenvalues with negative real parts, we see that $\sptn$ is an inertially arbitrary pattern.
\end{proof}

\section{Concluding remarks}

The idea of the bifurcation lemma was motivated by some of the work in \cite{Arnold71}, and the same condition as the nSSP is mentioned in \cite{Arnold71} for the study of versal deformation and in \cite{Laffey98} for extreme nonnegative matrices.  
In most of recent studies of the strong properties, e.g., \cite{gSAP, IEPG2}, the implicit function theorem is used for  building a new symmetric matrix with additional nonzero entries.  The technique that we have here employed to prove the bifurcation lemma is to use the inverse function theorem instead, without creating more nonzero entries.  As we have seen, the notion of bifurcation works well with various strong properties and applies for both the symmetric and the non-symmetric cases.  It is foreseeable that this technique can be used in other circumstances, such as new strong properties, matrices with zero diagonal, and so on.

\section{Acknowledgements}
Shaun M. Fallat was supported in part by an NSERC Discovery Research Grant, Application No.: RGPIN--2019--03934.
Jephian C.-H. Lin was supported by the Young Scholar Fellowship Program (grant no.\ MOST-110-2628-M-110-003) from the Ministry of Science and Technology in Taiwan.



\end{document}